\documentclass[draft,reqno]{amsproc}
\usepackage{amssymb}
\usepackage{euscript}
\usepackage{mathrsfs} %\mathscr T
\usepackage{units}   %\nicefrac %kolejarski ulamek
\usepackage{color}

\makeatletter
\@namedef{subjclassname@2010}{%
\textup{2010} Mathematics Subject Classification}
\makeatother

\numberwithin{equation}{section}

\newtheorem{thm}{Theorem}[section]

\newtheorem{lem}[thm]{Lemma}
\newtheorem{pro}[thm]{Proposition}

\newtheorem*{thm*}{Theorem}

\theoremstyle{remark}
\newtheorem{rem}[thm]{Remark}

\theoremstyle{definition}

\DeclareMathOperator{\D}{d\hspace{-0.25ex}}

\newcommand*{\ascr}{\mathscr A}

\newcommand*{\borel}[1]{{\mathfrak B}(#1)}

\newcommand*{\cbb}{\mathbb C}

\newcommand*{\esf}{\mathsf{E}}

\newcommand*{\dz}[1]{{\EuScript D}(#1)}

\newcommand*{\hh}{\mathcal H}

\newcommand*{\hsf}{\mathsf h}

\newcommand*{\Le}{\leqslant}

\newcommand*{\nbb}{\mathbb N}

\newcommand*{\nsf}{\mathsf{N}}

\newcommand*{\pscr}{{\mathscr P}}

\newcommand*{\rbb}{\mathbb R}
\newcommand*{\rbop}{{\overline{\mathbb R}_+}}

\newcommand*{\smalloplus}{\raise0pt
\hbox{$\scriptscriptstyle \oplus$}}

\newcommand*{\zbb}{\mathbb Z}

\begin{document}
   \title[Quasinormal composition
operators in $L^2$-spaces]{A multiplicative property
characterizes quasinormal composition operators in
$L^2$-spaces}
   \author[P.\ Budzy\'{n}ski]{Piotr Budzy\'{n}ski}
   \address{Katedra Zastosowa\'{n} Matematyki,
Uniwersytet Rolniczy w Krakowie, ul.\ Balicka 253c,
PL-30198 Krak\'ow, Poland}
   \email{piotr.budzynski@ur.krakow.pl}
   \author[Z.\ J.\ Jab{\l}o\'nski]{Zenon Jan
Jab{\l}o\'nski}
   \address{Instytut Matematyki,
Uniwersytet Jagiello\'nski, ul.\ \L ojasiewicza 6,
PL-30348 Kra\-k\'ow, Poland}
   \email{Zenon.Jablonski@im.uj.edu.pl}
   \author[I.\ B.\ Jung]{Il Bong Jung}
   \address{Department of Mathematics, Kyungpook National
University, Daegu 702-701, Korea}
   \email{ibjung@knu.ac.kr}
   \author[J.\ Stochel]{Jan Stochel}
\address{Instytut Matematyki, Uniwersytet
Jagiello\'nski, ul.\ \L ojasiewicza 6, PL-30348
Kra\-k\'ow, Poland}
   \email{Jan.Stochel@im.uj.edu.pl}
   \thanks{The work of the first, the second
and the fourth authors was supported by the MNiSzW
(Ministry of Science and Higher Education) grant NN201
546438 (2010-2013). The third author was supported by
Basic Science Research Program through the National
Research Foundation of Korea (NRF) funded by the
Ministry of Education, Science and Technology
(2012-008590).}
   \begin{abstract}
A densely defined composition operator in an
$L^2$-space induced by a measurable transformation
$\phi$ is shown to be quasinormal if and only if the
Radon-Nikodym derivatives $\hsf_{\phi^n}$ attached to
powers $\phi^n$ of $\phi$ have the multiplicative
property:\ $\hsf_{\phi^n} = \hsf_{\phi}^n$ almost
everywhere for $n = 0, 1, 2, \ldots$.
   \end{abstract}
   \maketitle
   \section{Introduction}
Composition operators (in $L^2$-spaces over
$\sigma$-finite measure spaces) play an essential role
in ergodic theory. They are also interesting objects
of operator theory. The foundations of the theory of
bounded composition operators are well-developed. In
particular, the questions of their boundedness,
normality, quasinormality, subnormality, seminormality
etc.\ were answered (see e.g.,
\cite{sin,nor,wh,ha-wh,lam1,lam2,di-ca,emb-lam3,sin-man,bu-ju-la}
for the general approach and
\cite{emb-lam2,ml,sto,da-st,2xSt} for special classes
of operators; see also the monograph \cite{sin-man}).

As opposed to the bounded case, the theory of
unbounded composition operators is at a rather early
stage of development. There are few papers concerning
this issue. Some basic facts about unbounded
composition operators can be found in
\cite{ca-hor,jab,b-j-j-sC}. In a recent paper
\cite{b-j-j-sD}, we gave the first ever criterion for
subnormality of unbounded densely defined composition
operators, which states that if such an operator
admits a measurable family of probability measures
that satisfy the consistency condition (see
\eqref{consist10.6}), then it is subnormal (cf.\
\cite[Theorem 3.5]{b-j-j-sD}). The aforesaid criterion
becomes a full characterization of subnormality in the
bounded case. Recall that the celebrated Lambert's
characterization of subnormality of bounded
composition operators (cf.\ \cite{lam1}) is no longer
true for unbounded ones (see \cite[Theorem
4.3.3]{j-j-s0} and \cite[Conclusion 10.5]{b-j-j-sC}).
It turns out that the consistency condition is
strongly related to quasinormality.

Quasinormal operators, which were introduced by A.
Brown in \cite{bro}, form a class of operators which
is properly larger than that of normal operators, and
properly smaller than that of subnormal operators (see
\cite[Theorem 1]{bro} and \cite[Theorem 2]{StSz2}). It
was A. Lambert who noticed that if $C_{\phi}$ is a
bounded quasinormal composition operator with a
surjective symbol $\phi$, then the Radon-Nikodym
derivatives $\hsf_{\phi^n}$, $n = 0, 1, 2, \ldots$,
(see \eqref{hfi}) have the following multiplicative
property (cf.\ \cite[p.\ 752]{lam1}):
   \begin{align*}
\hsf_{\phi^n} = \hsf_{\phi}^n \;\; \text{almost
everywhere for $n = 0, 1, 2, \ldots$.}
   \end{align*}
The aim of this article is to show that the above
completely characterizes quasinormal composition
operators regardless of whether they are bounded or
not, and regardless of whether $\phi$ is surjective or
not (cf.\ Theorem \ref{chquas}). The proof of this
characterization depends on the fact that a
quasinormal composition operator always admits a
special measurable family of probability measures
which satisfy the consistency condition
\eqref{consist10.6}. This leads to yet another
characterization of quasinormality (see condition
(iii) of Theorem \ref{chquas}).
   \section{Preliminaries}
We write $\cbb$ for the field of all complex numbers
and denote by $\rbb_+$, $\zbb_+$ and $\nbb$ the sets
of nonnegative real numbers, nonnegative integers and
positive integers, respectively. Set $\rbop = \rbb_+
\cup \{\infty\}$. Given a sequence
$\{\varDelta_n\}_{n=1}^\infty$ of sets and a set
$\varDelta$ such that $\varDelta_n \subseteq
\varDelta_{n+1}$ for every $n\in \nbb$, and $\varDelta
= \bigcup_{n=1}^\infty \varDelta_n$, we write
$\varDelta_n \nearrow \varDelta$ (as $n\to \infty$).
The characteristic function of a set $\varDelta$ is
denoted by $\chi_\varDelta$ (it is clear from the
context on which set the function $\chi_\varDelta$ is
defined).

The following lemma is a direct consequence of
\cite[Proposition I-6-1]{Nev} and \cite[Theorem
1.3.10]{Ash}. It will be used in the proof of Theorem
\ref{chquas}.
   \begin{lem} \label{2miary}
Let $\pscr$ be a semi-algebra of subsets of a set $X$
and $\rho_1, \rho_2$ be finite measures\footnote{\;All
measures considered in this paper are assumed to be
positive.} defined on the $\sigma$-algebra generated
by $\pscr$ such that $\rho_1(\varDelta) =
\rho_2(\varDelta)$ for all $\varDelta \in \pscr$. Then
$\rho_1 = \rho_2$.
   \end{lem}
Let $A$ be a linear operator in a complex Hilbert
space $\hh$. Denote by $\dz{A}$ and $A^*$ the domain
and the adjoint of $A$ (in case it exists). If $A$ is
closed and densely defined, then $A$ has a (unique)
{\em polar decomposition} $A=U|A|$, where $U$ is a
partial isometry on $\hh$ such that the kernels of $U$
and $A$ coincide and $|A|$ is the square root of
$A^*A$ (cf.\ \cite[Section 8.1]{b-s}). A densely
defined linear operator $A$ in $\hh$ is said to be
{\em quasinormal} if $A$ is closed and $U |A|
\subseteq |A|U$, where $A=U|A|$ is the polar
decomposition of $A$. We refer the reader to
\cite{bro} and \cite{StSz2} for basic information on
bounded and unbounded quasinormal operators,
respectively.

Throughout the paper $(X,\ascr, \mu)$ will denote a
$\sigma$-finite measure space. We shall abbreviate the
expressions ``almost everywhere with respect to
$\mu$'' and ``for $\mu$-almost every $x$'' to ``a.e.\
$[\mu]$'' and ``for $\mu$-a.e.\ $x$'', respectively.
As usual, $L^2(\mu)=L^2(X,\ascr, \mu)$ denotes the
Hilbert space of all square integrable complex
functions on $X$ with the standard inner product. Let
$\phi\colon X \to X$ be an $\ascr$-{\em measurable}
transformation of $X$, i.e., $\phi^{-1}(\varDelta) \in
\ascr$ for all $\varDelta \in \ascr$. Denote by
$\mu\circ \phi^{-1}$ the measure on $\ascr$ given by
$\mu\circ \phi^{-1}(\varDelta) =
\mu(\phi^{-1}(\varDelta))$ for $\varDelta \in \ascr$.
We say that $\phi$ is {\em nonsingular} if $\mu\circ
\phi^{-1}$ is absolutely continuous with respect to
$\mu$. If $\phi$ is a nonsingular transformation of
$X$, then the map $C_\phi\colon L^2(\mu) \supseteq
\dz{C_\phi}\to L^2(\mu)$ given by
   \begin{align*}
\dz{C_\phi} = \{f \in L^2(\mu) \colon f \circ \phi \in
L^2(\mu)\} \text{ and } C_\phi f = f \circ \phi \text{
for } f \in \dz{C_\phi},
   \end{align*}
is well-defined (and vice versa). Call such $C_\phi$ a
{\em composition operator}. Note that every
composition operator is closed (see e.g.,
\cite[Proposition 3.2]{b-j-j-sC}). If $\phi$ is
nonsingular, then by the Radon-Nikodym theorem there
exists a unique (up to sets of measure zero)
$\ascr$-measurable function $\mathsf \hsf_\phi \colon
X \to \rbop$ such that
   \begin{align} \label{hfi}
\mu\circ \phi^{-1}(\varDelta) = \int_\varDelta
\hsf_\phi \D\mu, \quad \varDelta \in \ascr.
   \end{align}
It is well-known that $C_{\phi}$ is densely defined if
and only if $\hsf_{\phi} < \infty$ a.e.\ $[\mu]$ (cf.\
\cite[Lemma 6.1]{ca-hor}), and
$\dz{C_{\phi}}=L^2(\mu)$ if and only if $\hsf_{\phi}
\in L^\infty(\mu)$ (cf.\ \cite[Theorem 1]{nor}). Given
$n \in \nbb$, we denote by $\phi^n$ the $n$-fold
composition of $\phi$ with itself; $\phi^0$ is the
identity transformation of $X$. Note that if $\phi$ is
nonsingular and $n\in \zbb_+$, then $\phi^n$ is
nonsingular and thus $\hsf_{\phi^n}$ makes sense.
Clearly $\hsf_{\phi^0}=1$ a.e.\ $[\mu]$.

Suppose that $\phi\colon X \to X$ is a nonsingular
transformation such that $\hsf_\phi < \infty$ a.e.\
$[\mu]$. Then the measure $\mu|_{\phi^{-1}(\ascr)}$ is
$\sigma$-finite (cf.\ \cite[Proposition
3.2]{b-j-j-sC}). Hence, by the Radon-Nikodym theorem,
for every $\ascr$-measurable function $f\colon X \to
\rbop$ there exists a unique (up to sets of measure
zero) $\phi^{-1}(\ascr)$-measurable function
$\esf(f)\colon X \to \rbop$ such that
   \begin{align} \label{przegad}
\int_{\phi^{-1}(\varDelta)} f \D\mu =
\int_{\phi^{-1}(\varDelta)} \esf(f) \D\mu, \quad
\varDelta \in \ascr.
   \end{align}
We call $\esf(f)$ the {\em conditional expectation} of
$f$ with respect to $\phi^{-1}(\ascr)$ (see
\cite{b-j-j-sC} for recent applications of the
conditional expectation in the theory of unbounded
composition operators; see also \cite{Rao} for the
foundations of the theory of probabilistic conditional
expectation). It is well-known that
   \begin{align} \label{CE-2}
\text{if $0\Le f_n \nearrow f$ and $f,f_n$ are
$\ascr$-measurable, then $\esf(f_n) \nearrow
\esf(f)$,}
   \end{align}
where $g_n \nearrow g$ means that for $\mu$-a.e.\
$x\in X$, the sequence $\{g_n(x)\}_{n=1}^\infty$ is
monotonically increasing and convergent to $g(x)$.

Now we state three results, each of which will be used
in the proof of Theorem \ref{chquas}. The first one
provides a necessary and sufficient condition for the
Radon-Nikodym derivatives $\hsf_{\phi^n}$, $n\in
\nbb$, to have the following semigroup property.
      \begin{lem}[\mbox{\cite[Lemma 9.1]{b-j-j-sC}}] \label{hf2}
If $\phi$ is a nonsingular transformation of $X$ such
that $\hsf_{\phi} < \infty$ a.e.\ $[\mu]$ and $n\in
\nbb$, then the following two conditions are
equivalent{\em :}
   \begin{enumerate}
   \item[(i)] $\hsf_{\phi^{n+1}} = \hsf_{\phi^n} \cdot \hsf_\phi$ a.e.\
$[\mu]$,
   \item[(ii)] $\esf(\hsf_{\phi^n}) = \hsf_{\phi^n}
\circ \phi$ a.e.\ $[\mu|_{\phi^{-1}(\ascr)}]$.
   \end{enumerate}
   \end{lem}
The second result is a basic description of
quasinormal composition operators.
   \begin{pro}[\mbox{\cite[Proposition 8.1]{b-j-j-sC}}]
\label{quasi} If $\phi\colon X \to X$ is nonsingular
and $C_{\phi}$ is densely defined, then $C_{\phi}$ is
quasinormal if and only if $\hsf_\phi = \hsf_\phi
\circ \phi$ a.e.\ $[\mu]$.
   \end{pro}
Before formulating the third result, we introduce the
necessary terminology. We say that a map $P\colon X
\times \borel{\rbb_+} \to [0,1]$, where
$\borel{\rbb_+}$ is the $\sigma$-algebra of all Borel
subsets of $\rbb_+$, is an {\em $\ascr$-measurable
family of probability measures} if the set-function
$P(x,\cdot)$ is a Borel probability measure on
$\rbb_+$ for every $x \in X$, and the function
$P(\cdot,\sigma)$ is $\ascr$-measurable for every
$\sigma \in \borel{\rbb_+}$. Let $\phi$ be a
nonsingular transformation of $X$ such that
$\hsf_{\phi} < \infty$ a.e.\ $[\mu]$. An
$\ascr$-measurable family $P\colon X \times
\borel{\rbb_+} \to [0,1]$ of probability measures is
said to satisfy the {\em consistency condition} (cf.\
\cite{b-j-j-sD}) if
   \begin{align}  \tag{CC} \label{consist10.6}
\esf(P(\cdot, \sigma)) (x) = \frac{\int_{\sigma} t
P(\phi(x),\D t)}{\hsf_\phi(\phi(x))} \text{ for
$\mu$-a.e.\ $x \in X$ and every $\sigma \in
\borel{\rbb_+}$.}
   \end{align}
As shown in \cite[Proposition 3.6]{b-j-j-sD}, each
quasinormal composition operator $C_{\phi}$ has an
$\ascr$-measurable family $P\colon X \times
\borel{\rbb_+} \to [0,1]$ of probability measures
which satisfies the consistency condition
\eqref{consist10.6}. In fact, such $P$ can always be
chosen to be $\phi^{-1}(\ascr)$-measurable. As already
mentioned, the consistency condition
\eqref{consist10.6} leads to a criterion for
subnormality of unbounded composition operators (cf.\
\cite[Theorem 3.5]{b-j-j-sD}).

The third result relates moments of an
$\ascr$-measurable family $P$ of probability measures
satisfying \eqref{consist10.6} to Radon-Nikodym
derivatives $\hsf_{\phi^n}$, $n\in \zbb_+$.
   \begin{thm}[\mbox{\cite[Theorem 5.4]{b-j-j-sD}}] \label{sms}
Let $\phi$ be a nonsingular transformation of $X$ such
that $0 < \hsf_{\phi} < \infty$ a.e.\ $[\mu]$, and
$P\colon X \times \borel{\rbb_+} \to [0,1]$ be an
$\ascr$-measurable family of probability measures
which satisfies \eqref{consist10.6}. Then
   \begin{align*}
\hsf_{\phi^n}(x)=\int_{\rbb_+} t^n P(x,\D t) \text{
for $\mu$-a.e.\ $x \in X$ and for every $n \in
\zbb_+$.}
   \end{align*}
   \end{thm}
   \section{The characterization}
In this section we provide the main characterization
of quasinormal composition operators (see (v) below).
We begin by noting that if $C_{\phi}$ is quasinormal,
then the $\ascr$-measurable family $P\colon X \times
\borel{\rbb_+} \to [0,1]$ of probability measures
given by $P(x,\sigma) = \chi_{\sigma}(\hsf_{\phi}(x))$
for $x\in X$ and $\sigma \in \borel{\rbb_+}$ satisfies
the consistency condition \eqref{consist10.6} (of
course, under the assumption that $\hsf_{\phi}$ is
finite). The consistency condition written for this
particular $P$ appears in (iii) below. It is an
essential component of the proof of the
characterization.

From now on, we adhere to the convention $\infty^0=1$.
   \begin{thm} \label{chquas}
Let $\phi$ be a nonsingular transformation of $X$ such
that $C_{\phi}$ is densely defined. Then the following
six conditions are equivalent\,\footnote{\;Since
$\hsf_{\phi} < \infty$ a.e.\ $[\mu]$, the expressions
$f\circ \hsf_{\phi}$ and $f\circ \hsf_{\phi}\circ
\phi$ appearing in (iv) are defined a.e.\ $[\mu]$. To
overcome this disadvantage, one can simply set
$f(\infty)=0$.}{\em :}
   \begin{enumerate}
   \item[(i)] $C_{\phi}$ is quasinormal,
   \item[(ii)] $\chi_{\sigma} \circ \hsf_{\phi} \circ \phi
\cdot \chi_{\sigma} \circ \hsf_{\phi} = \chi_{\sigma}
\circ \hsf_{\phi} \circ \phi$ a.e.\ $[\mu]$ for every
$\sigma \in \borel{\rbb_+}$,
   \item[(iii)] $\esf(\chi_{\sigma} \circ \hsf_{\phi})
= \chi_{\sigma} \circ \hsf_{\phi} \circ \phi$ a.e.\
$[\mu]$ for every $\sigma \in \borel{\rbb_+}$,
   \item[(iv)] $\esf(f \circ \hsf_{\phi})
= f \circ \hsf_{\phi} \circ \phi$ a.e.\ $[\mu]$ for
every Borel function $f\colon \rbb_+ \to \rbop$,
   \item[(v)] $\hsf_{\phi^n} = \hsf_{\phi}^n$
a.e.\ $[\mu]$ for every $n\in \zbb_+$,
   \item[(vi)] $\esf(\hsf_{\phi}) = \hsf_{\phi}
\circ \phi$ a.e.\ $[\mu]$ and $\esf(\hsf_{\phi^n}) =
\esf(\hsf_{\phi})^n$ a.e.\ $[\mu]$ for every $n\in
\zbb_+$.
   \end{enumerate}

   \end{thm}
   \begin{proof}
It follows from \cite[Lemma 6.1]{ca-hor} that
$\hsf_\phi < \infty$ a.e.\ $[\mu]$.

(i)$\Rightarrow$(iii) Since, by Proposition
\ref{quasi}, $\hsf_{\phi} = \hsf_{\phi}\circ \phi$
a.e.\ $[\mu]$, we deduce that $\chi_{\sigma} \circ
\hsf_{\phi} = \chi_{\sigma} \circ \hsf_{\phi} \circ
\phi$ a.e.\ $[\mu]$ for every $\sigma \in
\borel{\rbb_+}$, which implies (iii).

(iii)$\Rightarrow$(iv) Since each Borel function
$f\colon \rbb_+ \to \rbop$ is a pointwise limit of an
increasing sequence of nonnegative Borel simple
functions, one can show that (iii) implies (iv) by
applying the Lebesgue monotone convergence theorem as
well as the additivity and the monotone continuity of
the conditional expectation (see \eqref{CE-2}).

(iv)$\Rightarrow$(iii) Obvious.

(iii)$\Rightarrow$(ii) In view of (iii) and
\eqref{przegad}, we have
   \begin{align} \label{weneed}
\int_{\phi^{-1}(\varDelta)} \chi_{\sigma} \circ
\hsf_{\phi} \D \mu = \int_{\phi^{-1}(\varDelta)}
\chi_{\sigma} \circ \hsf_{\phi} \circ \phi \D \mu,
\quad \varDelta \in \ascr.
   \end{align}
By our assumptions on the measure $\mu$, there exists
a sequence $\{X_n\}_{n=1}^\infty \subseteq \ascr$ such
that $\{\mu(X_n)\}_{n=1}^\infty \subseteq \rbb_+$,
$X_n \nearrow X$ as $n \to \infty$ and
   \begin{align} \label{hff-0}
\hsf_{\phi} \Le k \text{ a.e.\ $[\mu]$ on $X_k$ for
every $k\in \nbb$.}
   \end{align}
Substituting $\varDelta = \hsf_{\phi}^{-1}(\sigma)
\cap X_n$ into \eqref{weneed}, we see that the
following equality holds for all $n\in\nbb$ and
$\sigma \in \borel{\rbb_+}$,
   \begin{align*}
\mu(\phi^{-1}(X_n) \cap (\hsf_{\phi} \circ
\phi)^{-1}(\sigma) \cap \hsf_{\phi}^{-1}(\sigma)) =
\mu(\phi^{-1}(X_n) \cap (\hsf_{\phi} \circ
\phi)^{-1}(\sigma)) < \infty.
   \end{align*}
Hence for all $n\in\nbb$ and $\sigma \in
\borel{\rbb_+}$,
   \begin{align*}
\chi_{\phi^{-1}(X_n)} \cdot \chi_{(\hsf_{\phi} \circ
\phi)^{-1}(\sigma)} \cdot
\chi_{\hsf_{\phi}^{-1}(\sigma)} =
\chi_{\phi^{-1}(X_n)} \cdot \chi_{(\hsf_{\phi} \circ
\phi)^{-1}(\sigma)} \text{ a.e.\ $[\mu]$.}
   \end{align*}
Since $\phi^{-1}(X_n) \nearrow X$ as $n \to \infty$,
we get (ii).

(ii)$\Rightarrow$(i) Substituting $\rbb_+\setminus
\sigma$ into (ii) in place of $\sigma$, we get
   \begin{align*}
\chi_{\sigma} \circ \hsf_{\phi} = \chi_{\sigma} \circ
\hsf_{\phi} \circ \phi \cdot \chi_{\sigma} \circ
\hsf_{\phi} \overset{\mathrm{(ii)}} = \chi_{\sigma}
\circ \hsf_{\phi} \circ \phi \text{ a.e.\ $[\mu]$ for
every $\sigma \in \borel{\rbb_+}$.}
   \end{align*}
Applying the standard measure-theoretic argument, we
deduce that $f \circ \hsf_{\phi} = f \circ \hsf_{\phi}
\circ \phi$ a.e.\ $[\mu]$ for every Borel function $f
\colon \rbb_+ \to \rbop$. Substituting $f(t)=t$, $t
\in \rbb_+$, we see that $\hsf_{\phi} =
\hsf_{\phi}\circ \phi$ a.e.\ $[\mu]$. By Proposition
\ref{quasi}, this yields (i).

Summarizing, we have proved that the conditions (i) to
(iv) are equivalent.

(v)$\Rightarrow$(vi) Since $\hsf_{\phi^2} =
\hsf_{\phi}^2$ a.e.\ $[\mu]$, we infer from Lemma
\ref{hf2} that $\esf(\hsf_{\phi}) = \hsf_{\phi} \circ
\phi$ a.e.\ $[\mu]$. Clearly, $\hsf_{\phi^{n+1}} =
\hsf_{\phi}^n \cdot \hsf_{\phi} = \hsf_{\phi^n} \cdot
\hsf_{\phi}$ a.e.\ $[\mu]$ for every $n\in \zbb_+$.
Applying Lemma \ref{hf2} again, we deduce that
   \begin{align*}
\esf(\hsf_{\phi}^n) = \esf(\hsf_{\phi^n}) =
\hsf_{\phi^n} \circ \phi = (\hsf_{\phi} \circ \phi)^n
= \esf(\hsf_{\phi})^n \text{ a.e.\ $[\mu]$ for every
$n \in \zbb_+$.}
   \end{align*}

(vi)$\Rightarrow$(v) Plainly, the equality
$\hsf_{\phi^n} = \hsf_{\phi}^n$ a.e.\ $[\mu]$ holds
for $n=0$. Suppose that it is valid for a fixed $n\in
\zbb_+$. Then
   \begin{align*}
\esf(\hsf_{\phi^n}) = \esf(\hsf_{\phi})^n =
\hsf_{\phi}^n \circ \phi = \hsf_{\phi^n} \circ \phi
\text{ a.e.\ $[\mu]$.}
   \end{align*}
This together with Lemma \ref{hf2} gives
   \begin{align*}
\hsf_{\phi^{n+1}} = \hsf_{\phi^n} \cdot \hsf_{\phi} =
\hsf_{\phi}^n \cdot \hsf_{\phi} = \hsf_{\phi}^{n+1}
\text{ a.e.\ $[\mu]$.}
   \end{align*}

(i)$\Rightarrow$(v) Without loss of generality, we may
assume that $\hsf_{\phi}(x) < \infty$ for all $x\in
X$. Note that $\hsf_{\phi} \circ \phi > 0$ a.e.\
$[\mu]$ (because $\phi^{-1}(\nsf_\phi) = \{x \in X
\colon \hsf_\phi (\phi (x))=0\}$ and
$\mu(\phi^{-1}(\nsf_\phi)) = \int_X \chi_{\nsf_\phi}
\circ \phi \D \mu = \int_X \chi_{\nsf_\phi} \hsf_\phi
\D \mu = 0$ with $\nsf_\phi:=\{x \in X \colon
\hsf_\phi(x)=0\}$). Hence, by Proposition \ref{quasi},
we have $\hsf_{\phi} = \hsf_{\phi} \circ \phi > 0$
a.e.\ $[\mu]$. Let $P\colon X \times \borel{\rbb_+}
\to [0,1]$ be the $\phi^{-1}(\ascr)$-measurable family
of probability measures defined by
   \begin{align}  \label{quas1}
P(x,\sigma) = \chi_{\sigma}(\hsf_{\phi}(\phi(x))),
\quad x\in X, \sigma \in \borel{\rbb_+}.
   \end{align}
Then $P$ satisfies \eqref{consist10.6} (see e.g., the
proof of \cite[Proposition 3.6]{b-j-j-sD}). Hence, by
Theorem \ref{sms}, we see that for every $n\in
\zbb_+$,
   \begin{align*}
\hsf_{\phi^n}(x) = \int_{\rbb_+} t^n P(x,\D t)
\overset{\eqref{quas1}}= (\hsf_{\phi} \circ \phi)^n
(x) = \hsf_{\phi}^n (x) \text{ for $\mu$-a.e.\ $x\in
X$.}
   \end{align*}

(v)$\Rightarrow$(iii) Let $\{X_n\}_{n=1}^\infty$ be as
in the proof of (iii)$\Rightarrow$(ii). By
\eqref{hff-0} and the nonsingularity of $\phi$, we
have
   \begin{align} \label{wineq}
\text{$\hsf_{\phi} \circ \phi \Le k$ a.e.\ $[\mu]$ on
$\phi^{-1}(X_k)$ for every $k\in \nbb$.}
   \end{align}
Now we show that
   \begin{align} \label{hff-1}
\hsf_{\phi} \Le k \text{ a.e.\ $[\mu]$ on
$\phi^{-1}(X_k)$ for every $k \in \nbb$.}
   \end{align}
For this, note that by (v) and the measure transport
theorem
   \begin{multline} \label{wzor}
\int_{\phi^{-1}(\varDelta)} \hsf_{\phi}^n \D \mu =
\int_X \chi_{\varDelta} \circ \phi \cdot \hsf_{\phi^n}
\D \mu = \int_{X} \chi_{\varDelta} \circ \phi^{n+1} \D
\mu
   \\
= \int_{\varDelta} \hsf_{\phi^{n+1}} \D \mu =
\int_{\varDelta} \hsf_{\phi}^{n+1} \D \mu, \quad
\varDelta \in \ascr, \, n\in \nbb.
   \end{multline}
Substituting $\varDelta = X_k$ into \eqref{wzor} and
using \eqref{hff-0}, we obtain
   \begin{align} \label{szyb}
\Big(\int_{\phi^{-1}(X_k)} \hsf_{\phi}^n \D
\mu\Big)^{1/n} \Le (k^{n+1} \mu(X_k))^{1/n}, \quad k,n
\in \nbb.
   \end{align}
By \cite[p.\ 95, Problem 9]{Ash}, this implies
\eqref{hff-1}.

It follows from (vi) that $\esf(\hsf_{\phi^n}) =
\hsf_{\phi}^n \circ \phi$ a.e.\ $[\mu]$ for all $n\in
\zbb_+$. Hence, by \eqref{przegad},
   \begin{align} \label{goni}
\int_{\phi^{-1}(\varDelta)} \hsf_{\phi^n} \D \mu =
\int_{\phi^{-1}(\varDelta)} \hsf_{\phi}^n \circ \phi
\D \mu
   \end{align}
for all $\varDelta \in \ascr$ and $n \in \zbb_+$. Fix
$k\in \nbb$ and $\varDelta \in \ascr$ such that
$\varDelta \subseteq X_k$. In view of \eqref{wineq}
and \eqref{hff-1}, there exists $E \in \ascr$ such
that $E \subseteq \phi^{-1}(\varDelta)$,
$\mu(\phi^{-1}(\varDelta) \setminus E)=0$ and
   \begin{align} \label{hff-2}
\hsf_{\phi}(x), \hsf_{\phi}(\phi(x)) \in [0,k] \quad
\text{for every $x\in E$.}
   \end{align}
By \eqref{szyb}, both sides of \eqref{goni} are finite
for every $n\in \zbb_+$. Therefore, we have
   \begin{align} \label{pph}
\int_{E} p \circ \hsf_{\phi} \D \mu = \int_{E} p \circ
\hsf_{\phi} \circ \phi \D \mu, \quad p \in \cbb[z],
   \end{align}
where $\cbb[z]$ is the ring of all complex polynomials
in variable $z$. Let $f\colon [0,k] \to \cbb$ be a
continuous function. By Weierstrass theorem, there
exists a sequence $\{p_n\}_{n=1}^\infty \subseteq
\cbb[z]$ which is convergent uniformly to $f$ on
$[0,k]$. Then, by \eqref{hff-2}, both suprema
$\sup_{E}|f \circ \hsf_{\phi} \circ \phi - p_n \circ
\hsf_{\phi} \circ \phi|$ and $\sup_{E}|f \circ
\hsf_{\phi} - p_n \circ \hsf_{\phi}|$ are less than or
equal to $\sup_{[0,k]}|f - p_n|$ for every $n\in
\nbb$. Since, by \eqref{hff-0}, $\mu(E) \Le
\mu(\phi^{-1}(X_k))< \infty$, we deduce that $\int_{E}
p_n \circ \hsf_{\phi} \circ \phi \D \mu$ tends to
$\int_{E} f \circ \hsf_{\phi} \circ \phi \D \mu$ as $n
\to \infty$, and $\int_{E} p_n \circ \hsf_{\phi} \D
\mu$ tends to $\int_{E} f \circ \hsf_{\phi} \D \mu$ as
$n \to \infty$. Hence, by \eqref{pph}, for every
continuous function $f\colon [0,k] \to \cbb$,
   \begin{align} \label{pph2}
\int_{E} f \circ \hsf_{\phi} \D \mu = \int_{E} f \circ
\hsf_{\phi} \circ \phi \D \mu.
   \end{align}
Take an interval $J=[a,b)$ with $a,b \in \rbb_+$. Then
there exists a sequence of continuous functions $f_n
\colon [0,k] \to [0,1]$, $n \in \nbb$, which converges
to $\chi_{J\cap [0,k]}$ pointwise. Therefore, by
\eqref{hff-2}, $f_n \circ \hsf_{\phi} \circ \phi$
tends to $\chi_{J\cap [0,k]} \circ \hsf_{\phi} \circ
\phi$ pointwise on $E$ as $n\to \infty$, and $f_n
\circ \hsf_{\phi}$ tends to $\chi_{J\cap [0,k]} \circ
\hsf_{\phi}$ pointwise on $E$ as $n\to \infty$. This
combined with \eqref{pph2} and the Lebesgue dominated
convergence theorem shows that the equality
   \begin{align}  \label{nmc}
\int_{\phi^{-1}(\varDelta)} \chi_\sigma \circ
\hsf_{\phi} \D \mu = \int_{\phi^{-1}(\varDelta)}
\chi_\sigma \circ \hsf_{\phi} \circ \phi \D \mu
   \end{align}
holds for $\sigma=J \cap [0,k]$. Applying Lemma
\ref{2miary} to the Borel measures on $[0,k]$ (with
respect to $\sigma$) defined by the left-hand and the
right-hand sides of \eqref{nmc}, and to the
semi-algebra $\pscr = \{[a,b)\cap [0,k]\colon a,b\in
\rbb_+\}$, we see that \eqref{nmc} holds for every
Borel set $\sigma \subseteq [0,k]$.

In view of the above, if $\varDelta \in \ascr$ and
$\sigma \in \borel{\rbb_+}$, then by the Lebesgue
monotone convergence theorem and the fact that $X_k
\nearrow X$ as $k \to \infty$, we have
   \begin{multline*}
\int_{\phi^{-1}(\varDelta)} \chi_\sigma \circ
\hsf_{\phi} \D \mu = \lim_{k \to \infty}
\int_{\phi^{-1}(\varDelta \cap X_k)} \chi_{\sigma \cap
[0,k]} \circ \hsf_{\phi} \D \mu
   \\
= \lim_{k \to \infty} \int_{\phi^{-1}(\varDelta \cap
X_k)} \chi_{\sigma \cap [0,k]} \circ \hsf_{\phi} \circ
\phi \D \mu = \int_{\phi^{-1}(\varDelta)} \chi_\sigma
\circ \hsf_{\phi} \circ \phi \D \mu,
   \end{multline*}
which together with \eqref{przegad} implies (iii).
This completes the proof.
   \end{proof}
   \begin{rem}
Regarding Theorem \ref{chquas}, it is worth pointing
out that it may happen that the equalities
$\hsf_{\phi^n} = \hsf_{\phi}^n$ a.e.\ $[\mu]$, $n\in
\zbb_+$, are satisfied though the equality
$\hsf_{\phi} = \hsf_{\phi} \circ \phi$ a.e.\ $[\mu]$
is not. This shows that the assumption
$\overline{\dz{C_{\phi}}} = L^2(\mu)$ in Theorem
\ref{chquas} is essential. To demonstrate this, set
$X=\zbb_+$ and $\ascr=2^X$. Let $\mu$ be the counting
measure on $\ascr$. Define the nonsingular
transformation $\phi$ of $X$ by $\phi(x) = 0$ for
$x\in X$. It is easily seen that $\hsf_{\phi^n}(x) =
0$ if $x \in X \setminus \{0\}$ and $\hsf_{\phi^n}(0)
= \infty$ for every $n\in \nbb$. This implies that
$\hsf_{\phi^n}(x) = \hsf_{\phi}(x)^n$ for all $x\in X$
and $n\in \zbb_+$. However, $\hsf_{\phi} \circ \phi(x)
= \infty$ for every $x\in X$, and $\hsf_{\phi}(x) = 0$
for every $x \in X \setminus \{0\}$.
   \end{rem}
   \bibliographystyle{amsalpha}
   
   \end{document}